\documentclass[a4paper,11pt]{amsart}

\usepackage[utf8]{inputenc}

\usepackage[T1]{fontenc}

\usepackage[english]{babel}

\usepackage{amsmath}
\usepackage{amsfonts}
\usepackage{amssymb}
\usepackage{amsthm} 
\usepackage{mathrsfs} 

\usepackage{subfig} 

\usepackage{tabularx} 
\newcolumntype{C}{>{\centering}X} 
\usepackage{calc} 
\usepackage[pdftex]{graphicx,color} 

\usepackage{multirow} 

\usepackage{xy} 
\usepackage{bm}
\xyoption{all}
\xyoption{poly}
\xyoption{arc}

\usepackage{a4wide} 

\usepackage{url} 
\usepackage{enumerate}

\usepackage{graphics}

\usepackage{array} 

\usepackage[hyperfootnotes]{hyperref}

\theoremstyle{plain} 
\newtheorem{theo}{Theorem}[section]
\newtheorem*{theo*}{Theorem} 
\newtheorem{coro}[theo]{Corollary}
\newtheorem{propo}[theo]{Proposition}
\newtheorem{lemma}[theo]{Lemma}

\theoremstyle{definition}
\newtheorem{defi}[theo]{Definition}

\newtheorem{remark}[theo]{Remark}

\newtheorem{example}[theo]{Example}

\newcommand{\FS}{{\mathfrak{S}}}

\newcommand{\CL}{{\mathcal{L}}} 
 \newcommand{\CH}{{\mathcal{H}}}
 \newcommand{\CR}{{\mathcal{R}}}
 
 \newcommand{\CA}{{\mathcal{A}}}

\renewcommand{\CR}{{\mathcal{R}}}
\renewcommand{\CL}{{\mathcal{L}}} 

\newcommand{\CK}{{\mathcal{K}}}

\newcommand{\CB}{{\mathcal{B}}}

\newcommand{\BN}{{\mathbb{N}}}
\newcommand{\BC}{{\mathbb{C}}} 

\newcommand{\BR}{{\mathbb{R}}}

\newcommand{\fp}{{\mathfrak{p}}}
\newcommand{\fq}{{\mathfrak{q}}}

\makeatletter
\def\hlinewd#1{%
\noalign{\ifnum0=`}\fi\hrule \@height #1 %
\futurelet\reserved@a\@xhline}
\makeatother

 \newenvironment{disarray}%
  {\everymath{\displaystyle\everymath{}}\array}%
  {\endarray}

\newcommand{\qg}{{\backslash}} 
\newcommand{\eps}{\varepsilon} 

\newcommand{\Lb}{\bar{\CL}}

\newcommand{\Vreg}{{V^{\mathrm{reg}}}}

\newcommand{\Enreg}{{E_n^{\mathrm{reg}}}}

\newcommand{\eh}{{e^{2i\pi/h}}}
\newcommand{\tq}{{~|~}}
\newcommand{\ie}{{i.e.}}

\newcommand{\surj}{{\twoheadrightarrow}}
\newcommand{\inj}{{\hookrightarrow}}

\newcommand{\slfrac}[2]{\left.#1\middle/#2\right.}

\newcommand{\DP}[3][]{\frac{\partial^{#1} #2}{\partial #3^{#1}}}

\newcommand{\GL}{\operatorname{GL}\nolimits}

\newcommand{\Red}{\operatorname{Red}\nolimits}

\newcommand{\Cat}{\operatorname{Cat}\nolimits}
\newcommand{\Spec}{\operatorname{Spec}\nolimits}
\newcommand{\Ker}{\operatorname{Ker}\nolimits}

\newcommand{\Disc}{\operatorname{Disc}\nolimits}

\renewcommand{\det}{\operatorname{det}\nolimits}

\renewcommand{\Red}{\operatorname{Red}\nolimits}

\newcommand{\Jac}{\operatorname{Jac}\nolimits}
\newcommand{\Fact}{\operatorname{\textsc{fact}}\nolimits}
\newcommand{\fact}{\operatorname{\underline{facto}}\nolimits}
\newcommand{\LL}{\operatorname{LL}\nolimits}

\newcommand{\Fix}{\operatorname{Fix}\nolimits}

\newcommand{\rk}{\operatorname{rk}\nolimits}

\newcommand{\NC}{\operatorname{NC}\nolimits}

\newcommand{\<}{\preccurlyeq}

\newcommand{\Spram}{\operatorname{Spec_1^{\mathrm{ram}}}\nolimits}

\newcommand{\grad}{\operatorname{grad}\nolimits}

\begin{document}

\title[Lyashko-Looijenga morphisms and factorizations of a Coxeter
element]{Lyashko-Looijenga morphisms and submaximal factorizations of
  a Coxeter element}

\author{Vivien Ripoll}
 \address{LaCIM, UQÀM, CP 8888, Succ. Centre-ville Montréal, QC, H3C
  3P8, Canada}
\email{vivien.ripoll@lacim.ca}

\begin{abstract}
  When $W$ is a finite reflection group, the noncrossing partition
  lattice $\NC(W)$ of type $W$ is a rich combinatorial object,
  extending the notion of noncrossing partitions of an $n$-gon. A
  formula (for which the only known proofs are case-by-case) expresses
  the number of multichains of a given length in $\NC(W)$ as a
  generalized Fu\ss-Catalan number, depending on the invariant degrees
  of $W$. We describe how to understand some specifications of this
  formula in a case-free way, using an interpretation of the chains of
  $\NC(W)$ as fibers of a Lyashko-Looijenga covering ($\LL$),
  constructed from the geometry of the discriminant hypersurface of
  $W$. We study algebraically the map $\LL$, describing the
  factorizations of its discriminant and its Jacobian. As byproducts,
  we generalize a formula stated by K. Saito for real reflection
  groups, and we deduce new enumeration formulas for certain
  factorizations of a Coxeter element of $W$.
\end{abstract}

\maketitle
\section*{Introduction}
\label{sec:intro}

Complex reflection groups are a natural generalization of finite real
reflection groups (that is, finite Coxeter groups realized in their
geometric representation). In this article, we consider a
well-generated complex reflection group $W$; the precise definitions
will be given in Sec.\ \ref{subsec:crg}.

The \emph{noncrossing partition lattice} of type~$W$, denoted
$\NC(W)$, is a particular subset of $W$, endowed with a partial
order~$\<$ called the absolute order (see definition below). When $W$
is a Coxeter group of type $A$, $\NC(W)$ is isomorphic to the poset of
noncrossing partitions of a set, studied by Kreweras
\cite{krew}. Throughout the last 15 years, this structure has been
generalized to finite Coxeter groups, first (Reiner \cite{reiner},
Bessis \cite{Bessisdual}, Brady-Watt \cite{BWordre2}), then to
well-generated complex reflection groups (see \cite{BessisKPi1}). It has
many applications in the algebraic understanding of the braid group of
a reflection group (via the construction of the \emph{dual braid
  monoid}, see \cite{Bessisdual,BessisKPi1}), and is also studied for
itself as a very rich combinatorial object (see Armstrong's memoir
\cite{Armstrong}).

In order to introduce the structure $\NC(W)$, we need several
definitions and notations (which will be detailed in Sec.\
\ref{sec:ncp}):

\begin{itemize}
\item the set $\CR$ of \emph{all} reflections of $W$;
\item the \emph{reflection length} (or absolute length) $\ell$ on $W$: for
  $w$ in $W$, $\ell(w)$ is the minimal length of a word on the
  alphabet $\CR$ that represents $w$;
\item a \emph{Coxeter element} $c$ in $W$;
\item the \emph{absolute order} $\<$ on $W$, defined as \[ u\< v \text{ if and
    only if } \ell(u) + \ell(u^{-1}v)=\ell(v).\]
\end{itemize}

The \emph{noncrossing partition lattice} associated to $(W,c)$
is defined to be the interval below $c$:
\[ \NC(W)= \{w\in W \tq w \< c \}.\]

This lattice has a fascinating combinatorics, and one of its most
amazing properties concerns its \emph{Zeta} polynomial (expressing the
number of multichains of a given length).

\textbf{``Chapoton's formula''}
\emph{Let $W$ be an irreducible, well-generated complex reflection group, of
  rank $n$. Then, for any $p\in \BN$, the number of multichains $w_1 \<
  \dots \< w_{p}$ in the poset $\NC(W)$ is equal~to 
  \[ \Cat^{(p)}(W) = \prod_{i=1}^n \frac{d_i + ph}{d_i} \ ,\] where
  $d_1\leq\dots \leq d_n=h$ are the invariant degrees of $W$ (defined
  in Sec.\ \ref{subsec:crg}).
}

\medskip

The numbers $\Cat^{(p)}(W)$ are called \emph{Fu\ss-Catalan numbers of
  type $W$} (and Catalan numbers for $p=1$). When $W$ is the symmetric
group $\FS_n$, these are the classical Catalan and Fu\ss-Catalan
numbers $\frac{1}{pn+1}\binom{(p+1)n}{n}$. Those generalized
Fu\ss-Catalan numbers also appear in other combinatorial objects
constructed from the group $W$, for example cluster algebras of finite
type introduced in \cite{FominZel} (see Fomin-Reading \cite{gcccc},
and the references therein).

In the real case, this formula was first stated by Chapoton in
\cite[Prop.\ 9]{chapoton}. The proof is case-by-case (using the
classification of finite Coxeter groups), and it mainly uses results
by Athanasiadis and Reiner \cite{reiner,athareiner} (see also
\cite{picantin:explicit}). The remaining complex cases are checked by
Bessis in \cite{BessisKPi1}, using results of
\cite{BessisCorran}. There is still no case-free proof of this
formula, even for the simplest case $p=1$, which states that the
cardinality of $\NC(W)$ is equal to the generalized Catalan number
\[\Cat(W)=\prod_{i=1}^n \frac{d_i + h}{d_i} \ .\]

This very simple formula naturally incites to look for a uniform proof
that could shed light on the mysterious relation between the
combinatorics of $\NC(W)$ and the invariant theory of $W$. This is the
problem which has motivated this work. Roughly speaking, we will bring
a complete geometric (and mainly case-free) understanding of certain
specifications of Chapoton's formula. For geometric reasons (that will
become clear in Sec.\ \ref{subsec:lbl}), we consider strict chains in
$\NC(W)$ of a given length, rather than multichains. In any bounded
posets, their numbers are related to the numbers of multichains by
well known conversion formulas: basically, they are the coefficients
of the \emph{Zeta} polynomial written in the basis of binomial
polynomials (see \cite[Ch.\ 3.11]{stanley}). An alternative way (more
adapted in our work) to look at strict chains in $\NC(W)$ is to
consider \emph{block factorizations} of the Coxeter element~$c$:

\begin{defi}
  \label{def:facto}
  For $c$ a Coxeter element of $W$, $(w_1,\dots, w_p)$ is called a
  \emph{block factorization} of~$c$ if:
  \begin{itemize} 
  \item $\forall i, \ w_i \in W - \{1\}$\ ;
  \item $w_1\dots w_p = c$\ ;
  \item $\ell(w_1)+\dots + \ell(w_p)=\ell(c)$.
  \end{itemize}
 \end{defi}
 
The reflection length of $c$ equals the rank of~$W$, denoted here
 by~$n$. Thus, the maximal number of factors in a block factorization
 is $n$. Note that block factorizations of $c$ have the same
 combinatorics of strict chains of $\NC(W)$: the partial products
 $w_1\dots w_i$, for~$i$ from $1$ to $p$, form a strict chain by
 definition. Thus, using simple computations as explained above,
 we can reformulate Chapoton's formula in terms of these
 factorizations (an explicit formula is given in Appendix B of
 \cite{These}). Proving Chapoton's formula amounts to computing the
 number of block factorizations in $p$ factors, for $p$ from $1$ to
 $n$.

 We call reduced decompositions of $c$ the factorizations of $c$ in
 $n$ reflections, \ie, the most refined block factorizations (the set
 of such factorizations is ususally denoted by $\Red_{\CR}(c)$). The
 reformulation implies in particular that the number of reduced
 decompositions (or, equivalently, the number of maximal strict chains
 in $\NC(W)$) is $n!$ times the leading coefficient of the \emph{Zeta}
 polynomial, that is:
 \[ |\Red_{\CR}(c)| = \frac{n!h^n}{|W|} \ . \]

 Note that this particular formula was known long before Chapoton's
 formula (the real case was dealt with by Deligne in
 \cite{deligneletter}; see \cite[Prop.\ 7.5]{BessisKPi1} for the
 remaining cases). Once again, even for this specific formula, no
 case-free proof is known.

 In \cite{BessisKPi1}, Bessis ---crediting discussions with
 Chapoton--- interpreted this integer $n!h^n/|W|$ as the degree of a
 covering (the \emph{Lyashko-Looijenga covering} $\LL$) constructed
 from the discriminant of $W$, and he described effectively the
 relations between the fibers of this covering and the reduced
 decompositions of~$c$. The aim of this paper is to explain how, by
 studying the map $\LL$ in more detail, we can obtain new enumerative
 results, namely formulas for the number of submaximal factorizations
 of~$c$.

 \begin{theo*}[see Thm.\ \ref{thm:fact} and Cor.\ \ref{cor:submax}]
  \label{thm:factintro}
  Let $W$ be an irreducible, well-generated complex reflection group,
  of rank $n$. Let $c$ be a Coxeter element of $W$ and $\Lambda$ be a
  conjugacy class of elements of reflection length $2$ in
  $\NC(W)$. Then:
  \begin{enumerate}[(a)]
  \item the number of block factorizations of $c$, made up with $n-2$
    reflections and one element in the conjugacy class $\Lambda$, is
    \[ |\Fact_{n-1}^{\Lambda}(c)| = \frac{(n-1)!\ h^{n-1}}{|W|}\ \deg
    D_\Lambda \ ,\] where $D_{\Lambda}$ is an homogeneous polynomial
    (in the $n-1$ first fundamental invariants) attached to $\Lambda$,
    determined by the geometry of the discriminant hypersurface of $W$
    (see Sec.\ \ref{sec:submax});
  \item the total number of block factorizations of $c$ in $n-1$
    factors (or submaximal factorizations) is 
    \[ |\Fact_{n-1}(c)| = \frac{(n-1)!\ h^{n-1}}{|W|}
    \left( \frac{(n-1)(n-2)}{2}h + \sum_{i=1}^{n-1} d_i \right) .\]
  \end{enumerate}
\end{theo*}

The first point is new, and is a refinement of the second which was
already known: like for the number of reduced decompositions, item (b)
is a consequence of Chapoton's formula. The main interest of stating
(b) is that the proof obtained here is geometric and almost case-free
(we still have to rely on some structural properties of $\LL$ proved
in \cite{BessisKPi1} case-by-case). The structure of the proof is
roughly as follows:
\begin{enumerate}
\item we use new geometric properties of the morphism $\LL$ to prove the
  formula of point (a) (Sec.\ \ref{subsec:submaxtype});
\item we find a uniform way to compute $\sum_{\Lambda} \deg D_{\Lambda}$,
  using an algebraic study of the Jacobian of $\LL$ (Sec.\ \ref{subsec:LLwell});
\item we deduce the second formula, since $|\Fact_{n-1}(c)|=\sum_{\Lambda}
  |\Fact_{n-1}^{\Lambda}(c)|$ (Sec.\ \ref{subsec:submax}).
\end{enumerate}

Thus, even if the method used here does not seem easily
generalizable to factorizations with fewer blocks, it is a new
interesting avenue towards a geometric case-free explanation of
Chapoton's formulas.

\begin{remark}
  During step (2) of the proof, we recover a formula proved
  (case-by-case) by K. Saito in \cite{saitopolyhedra} for real groups,
  and extend it for complex groups. This concerns the bifurcation
  locus of the discriminant hypersurface of $W$, the factorization of
  its equation, and the relation with the factorization of the
  Jacobian of $\LL$ (see Sec.\ \ref{subsec:saito}).
\end{remark}

\bigskip

\textbf{Outline.}  In Section \ref{sec:ncp} we give some backgrounds
and notations about complex reflection groups, the noncrossing
partition lattice, and block factorizations of a Coxeter
element. Section \ref{sec:LL} is devoted to the construction and
properties of the Lyashko-Looijenga covering of type $W$, and in
particular its relation with factorizations. 

Section \ref{sec:LLext} is the core of the proof: we study further the
algebraic properties of the morphism~$\LL$, we show that it gives rise
to a ``well-ramified'' polynomial extension, we derive factorizations
of its Jacobian and its discriminant into irreducibles. We also list
the analogies between the properties of $\LL$ extensions and those of
Galois extensions. In Section \ref{sec:submax} we use these results to
deduce the announced formulas for the number of submaximal
factorizations of a Coxeter element. We conclude in the last section
by giving a list of numerical data about these factorizations for each
irreducible well-generated complex reflection group.

\section{The noncrossing partition lattice of type $W$ and block
  factorizations of a Coxeter element}
\label{sec:ncp}
\subsection{Complex reflection groups}
~\medskip

\label{subsec:crg}

First we recall some notations and definitions about complex
reflection groups. For more details we refer the reader to the books
\cite{Kane} and \cite{unitary}.

For $V$ a finite dimensional complex vector space, we call a
\emph{reflection} of $\GL(V)$ an automorphism $r$ of $V$ of
finite order and such that the invariant space $\Ker(r-1)$ is a
hyperplane of~$V$ (it is called \emph{pseudo-reflection} by
some authors). We call a \emph{complex reflection group} a finite
subgroup of $\GL(V)$ generated by reflections.

A simple way to construct such a group is to take a finite real
reflection group (or, equivalently, a finite Coxeter group together
with its natural geometric realization) and to complexify it. There
are of course many other examples that cannot be seen in a real
space. A complete classification of irreducible complex reflection
groups was given by Shephard-Todd in \cite{Shephard}: it consists of
an infinite series with three parameters and $34$ exceptional groups
of small ranks.

\medskip

Throughout this paper we denote by $W$ a subgroup of $\GL(V)$ which is
a complex reflection group. Note that for real reflection groups the
results presented here are already interesting (and, most of them,
new).

We suppose that $W$ is irreducible of rank $n$ (\ie, the linear action
on $V$ is irreducible, and the dimension of $V$ is $n$). If
$(v_1,\dots, v_n)$ denotes a basis for $V$, $W$ acts naturally on the
polynomial algebra
$\BC[V]=\BC[v_1,\dots,v_n]$. Chevalley-Shephard-Todd's theorem implies
that the invariant algebra $\BC[V]^W$ is again a polynomial algebra,
and it can be generated by $n$ algebraically independent homogeneous
polynomials $f_1,\dots, f_n$ (called the \emph{fundamental
  invariants}). The degrees $d_1,\dots, d_n$ of these invariants do
not depend on the choices for the $f_i$'s (if we require $d_1 \leq
\dots \leq d_n$) and they are called the \emph{invariant degrees}
of~$W$. Like for finite Coxeter groups, we will denote by $h$ the
highest degree $d_n$ (called the Coxeter number of $W$).

We will also require that $W$ is \emph{well-generated} , \ie, it can be
generated by $n$ reflections (this is always verified in the real
case). Then there exist in $W$ so-called \emph{Coxeter elements},
which generalize the usual notion of a Coxeter element in finite Coxeter
groups.

\begin{defi}
  \label{def:coxelt}
  A \emph{Coxeter element}~$c$ of $W$ is an $\eh$-regular element (in
  the sense of Springer's regularity \cite{Springer}), \ie, it is such
  that there exists a vector $v$, outside the reflecting hyperplanes,
  such that $c(v)=\eh v$.
\end{defi}

As in the real case, Coxeter elements have reflection length $n$, and
form a conjugacy class of $W$.

\subsection{The noncrossing partition lattice of type $W$}
\label{subsec:ncp}
~\medskip

Recall that $\CR$ denotes the set of all reflections of $W$. For $w$
in $W$, the \emph{reflection length} (or \emph{absolute length}) of $w$ is:
\[ \ell(w)= \min \{p\in \BN \tq \exists r_1,\dots r_p \in \CR,\
w=r_1\dots r_p \} \ .\]

This length is not to be confused with the usual length in Coxeter
groups (called weak length, relative to the generating set of simple
reflections), which can be defined only in the real case.

The noncrossing partition lattice is constructed from the absolute
order, which is the natural prefix order for the reflection length:

\begin{defi}
  \label{def:order}
  We denote by $\<$ the absolute order on $W$, defined by:
  \[ u\< v \text{ if and only if } \ell(u) + \ell(u^{-1}v)=\ell(v).\] If
  $c$ is a Coxeter element of $W$, the \emph{noncrossing partition
    lattice} of $(W,c)$ is:
  \[ \NC(W;c)= \{w\in W \tq w \< c \}.\]
\end{defi}

Since all the Coxeter elements are conjugate, and the reflection
length is invariant under conjugation, the structure of $\NC(W;c)$
does not depend on the choice of the Coxeter element~$c$. Thus we will
just write $\NC(W)$ for short, considering $c$ fixed for the rest of
the paper. In the prototypal case of type $A$, where $W$ is the
symmetric group $\FS_{n+1}$, $\CR$ is the set of all transpositions
and $c$ is an $n+1$-cycle; then $\NC(W)$ is isomorphic to the set of
noncrossing partitions of an $n+1$-gon, as introduced by Kreweras in
\cite{krew}. In general, the noncrossing partition lattice of type $W$
has a very rich combinatorial structure: we refer to Chapter 1 of
\cite{Armstrong} or the introduction of \cite{These}.

\subsection{Multichains in $\NC(W)$ and block factorizations of a
  Coxeter element}
\label{subsec:multichains}
~\medskip

Recall from Def.\ \ref{def:facto} that a \emph{block factorization} of
$c$ is a factorization in nontrivial factors, such that the lengths of
the factors add up to the length of $c$ (\ie, it exists reduced
decompositions of $c$ obtained from concatenation of reduced
decompositions of the blocks).

We denote by $\Fact(c)$ (resp. $\Fact_p(c)$) the set of block
factorizations of $c$ (resp. factorizations in $p$ factors).  Note
that the length of $c$ is equal to the rank $n$ of $W$, so any block
factorization of~$c$ determines a \emph{composition} (ordered
partition) of the integer $n$. The set $\Fact_n(c)$ corresponds to the
set of reduced decompositions of $c$ into reflections, usually denoted
by~$\Red_{\CR}(c)$ (composition $(1,1,\dots, 1)$).

To simplify we will write, from now on, \emph{factorization} for block
factorization.

\medskip

If $(w_1,\dots, w_p)$ is a factorization of $c$, then we canonically
get a (strict) chain in $\NC(W)$:
\[ w_1 \prec w_1w_2 \prec \dots \prec w_1\dots w_p=c.\] Strict chains
are related to multichains by known formulas, so that we can pass from
enumeration of multichains in $\NC(W)$ to enumeration of factorizations of $c$, and vice versa (see for example \cite[App.\
B]{These} or \cite[Ch.\ 3.11]{stanley}).

\bigskip

In the following section, we describe a geometric construction of
these factorizations, and how they are related to the fibers of
a topological covering.

\section{Lyashko-Looijenga covering and factorizations of a 
Coxeter element}
\label{sec:LL}

\subsection{Discriminant of a well-generated reflection group and
  Lyashko-Looijenga covering}
\label{subsec:LL}
~\medskip

Let $W$ be a well-generated, irreducible complex reflection group, with
invariant polynomials $f_1,\dots, f_n$, homogeneous of degrees
$d_1\leq \dots \leq d_n=h$. Note that the quotient-space\footnote{The
  action of $W$ on $V$ is conventionally on the left side, so we
  prefer to write the quotient-space $W \qg V$.} $W \qg V$ is then
isomorphic to $\BC^n$:
\[\begin{array}{rcl}
W \qg V &\xrightarrow{\sim}& \BC^n\\
\bar{v} & \mapsto & (f_1(v),\dots, f_n(v))
\end{array}
\]

We recall here the construction of the Lyashko-Looijenga map of type
$W$ (for more details, see \cite[Sec.~5]{BessisKPi1} or \cite[Sec.\
3]{Ripollfacto}).

\medskip

Let us denote by $\CA$ the set of all reflecting hyperplanes of $W$,
and consider the discriminant of $W$ defined by
\[ \Delta_W := \prod_{H\in \CA} \alpha_H^{e_H} \ ,\] where $\alpha_H$
is an equation of $H$ and $e_H$ is the order of the parabolic subgroup
$W_H=\Fix(H)$. The discriminant lies in $\BC[V]^W=\BC[f_1,\dots,
f_n]$, and it is an equation for the discriminant hypersurface
\[ \CH:=W \ \qg \bigcup_{H\in \CA} H \ \subseteq W \qg V \simeq \BC^n \
.\]

It is known (see \cite[Thm.\ 2.4]{BessisKPi1}) that when $W$ is
well-generated, the fundamental invariants $f_1,\dots, f_n$ can be
chosen such that the discriminant of $W$ is a monic polynomial of
degree~$n$ in~$f_n$ of the form:
\[\Delta_W= f_n^n + a_2 f_n ^{n-2} +\dots + a_n \ ,\]
where $a_i \in \BC[f_1,\dots, f_{n-1}]$. This property implies that if
we fix $f_1, \dots, f_{n-1}$, then $\Delta_W$ always has $n$ roots
(counting multiplicities) as a polynomial in $f_n$.

Let us define $ Y:= \Spec \BC[f_1,\dots, f_{n-1}] \simeq \BC^{n-1} $,
so that $W \qg V \simeq Y \times \BC$.  Then the geometric version of
the property given above is that the intersection of the hypersurface
$\CH$ with the complex line $\{(y,f_n)\tq f_n\in \BC\}$ (for a fixed
$y\in Y$) generically has cardinality $n$. The definition of the
Lyashko-Looijenga map comes from these observations.

\begin{defi}
  We denote by $E_n$ be the set of centered configurations of $n$
  points in $\BC$, \ie,
  \[ E_n:= H_0/\FS_n \text{ , where }H_0=\Big\{(x_1,\dots, x_n)\in
    \BC^n ~\Big| ~ \sum_{i=1}^n x_i =0\Big\}.\] The Lyashko-Looijenga map
  of type $W$ is defined by:
  \[\begin{array}{rcl}
  Y & \xrightarrow{\LL} & E_n\\
    y=(f_1,\dots, f_{n-1}) & \mapsto & \text{multiset\ of\ roots\ of\ }
    \Delta_W(f_1,\dots, f_n) \text{ in the variable } f_n.
  \end{array}\]
\end{defi}

\begin{remark}
\label{rk:algLL}
We can also regard $\LL$ as an algebraic morphism. Indeed, the natural
coordinates for $E_n$ as an algebraic variety are the $n-1$ elementary
symmetric polynomials $e_2(x_1\dots, x_n), \dots, e_n(x_1,\dots,
x_n)$. Thus, the algebraic version of the map $\LL$ is (up to some
unimportant signs) simply the morphism
  \[\begin{array}{ccc}
    \BC^{n-1} & \to & \BC^{n-1}\\
    (f_1,\dots, f_{n-1}) & \mapsto & (a_2(f_1,\dots, f_{n-1})\ ,\
    \dots\ ,\  a_n(f_1, \dots, f_{n-1}))\ .
  \end{array} \]
  To shorten the notations, we will also denote this morphism by $\LL$ ,
  whenever in an algebraic context (mainly in Sec.\ \ref{sec:LLext}).
\end{remark}

We denote by $\Enreg$ the set of configurations in $E_n$ with $n$
distinct points, and we define the bifurcation locus of $\LL$, namely
$\CK:=\LL^{-1}(E_n-\Enreg)$. Equivalently, we have
\[\CK:= \{y \in Y \tq D_{\LL}(y)=0\},\]
where $D_{\LL}$ is called the $\LL$-discriminant and is defined by:
\[ D_{\LL} := \Disc(\Delta_W(y,f_n)\ ;\ f_n) \ \in \BC[f_1,\dots,
f_{n-1}].\]

\begin{example}
\label{ex:A3_0}
\begin{figure}[!h]
\centering

\resizebox{12cm}{!}{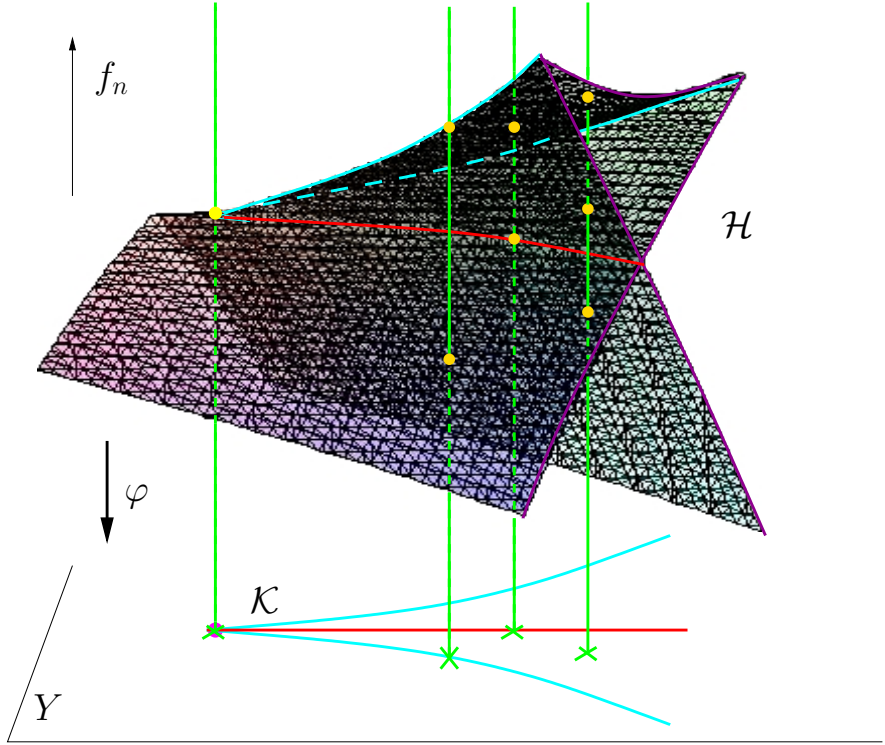}
\caption{Example of $W(A_3)$. The picture represents a fragment of the
  real part of the discriminant hypersurface $\CH$ (equation
  $\Disc(T^4+f_1T^2-f_2 T +f_3 \ ;\ T)=0$, called the swallowtail
  hypersurface), as well as its bifurcation locus $\CK$. The vertical
  is chosen to be the direction of $f_n$. The other information is
  described gradually in Examples \ref{ex:A3_0}, \ref{ex:A3}, and
  \ref{ex:A3_2}.}
\label{A3}
\end{figure}
  The picture of Fig.\ \ref{A3} gives a simplified geometric view of
  what happens for the group $W(A_3)$. The discriminant hypersurface
  $\CH$ and the bifurcation locus $\CK$ are described. The map $\LL$
  associates to any point in $Y$ the multiset of intersection points
  of the line $\{(y,x)\tq x \in \BC\}$ (vertical green lines )
  with~$\CH$ (yellow points).
\end{example}

The first important property is the following (from \cite[Thm.\
5.3]{BessisKPi1}):
\begin{equation}
\text{The restriction of }\LL: Y-\CK \surj \Enreg \text{ is a
  topological covering of degree }\frac{n!h^n}{|W|}\ .
\label{eq:degLL}
\tag{P0}
\end{equation}

We call this integer the \emph{Lyashko-Looijenga number of type $W$}.

\subsection{Geometric construction of factorizations}
~\medskip

\label{subsec:lbl}
Before explaining the construction of factorizations from the
discriminant hypersurface, we recall some useful properties of the
geometric stratification associated to the parabolic subgroups of $W$.
\medskip

\noindent \textbf{Discriminant stratification.}
The space $V$, together with the hyperplane arrangement $\CA$, admits
a natural stratification by the \emph{flats}, namely, the elements of the
intersection lattice $ \CL:=\left\{\bigcap_{H\in \CB} H \tq \CB
  \subseteq \CA \right\} $.

As the $W$-action on $V$ maps a flat to a flat, this stratification
gives rise to a quotient stratification $\Lb$ of $W \qg V$: \[\Lb= W
\qg \CL =(p(L))_{L \in \CL}= (W\cdot L)_{L \in \CL} \ ,\] where $p$ is
the projection $V \ \surj \ W \qg V $.  For each stratum $\Lambda$ in
$\Lb$, we denote by $\Lambda^0$ the complement in $\Lambda$ of the
union of the strata strictly included in $\Lambda$. The
family $(\Lambda^0)_{\Lambda \in \Lb}$ form an open stratification of $W \qg
V$, called the \emph{discriminant stratification}.

\medskip

There is a natural bijection between the set of flats in $V$ and the
set of parabolic subgroups of $W$ (Steinberg's theorem). By
quotienting by the action of $W$, this leads to
other descriptions of the stratification $\Lb$:

\begin{propo}
  \label{prop:parab}
  The set $\Lb$ is in canonical bijection with:
  \begin{itemize}
  \item the set of conjugacy classes of parabolic subgroups of $W$;
  \item the set of conjugacy classes of \emph{parabolic Coxeter elements}
    (\ie, Coxeter elements of parabolic subgroups);
  \item the set of conjugacy classes of elements of $\NC(W)$.
  \end{itemize}

  Through these bijections, the codimension of a stratum $\Lambda$
  corresponds to the rank of the associated parabolic subgroup and to
  the reflection length of the parabolic Coxeter element.
\end{propo}

We refer to \cite[Sec.\ 6]{Ripollfacto} for details and proofs.

\begin{example}
  \label{ex:A3}
  In the picture of Fig.\ \ref{A3}, the two strata of $\Lb$ of
  codimension 2 are drawn in red and blue (the blue one is the one
  forming a cusp). Through the bijection of Prop.\ \ref{prop:parab},
  the blue one corresponds to the conjugacy class of a parabolic
  Coxeter element of type $A_2$ (viewed in $\FS_4$, this is a
  $3$-cycle), and the red one corresponds to the conjugacy class of a
  parabolic Coxeter element of type $A_1\times A_1$ (\ie, a product of
  two commuting transpositions in $\FS_4$).
\end{example}

\medskip

\noindent \textbf{Geometric factorizations and compatibilities.}  In
\cite{Ripollfacto} we established a way to construct factorizations
geometrically from the discriminant hypersurface $\CH$. We describe
below the idea of the construction, and some of its properties; for
details and proofs, see \cite[Sec.\ 4]{Ripollfacto} and \cite[Sec.\
6]{BessisKPi1}.

The starting point is the construction of a map
\[\begin{array}{rrcl}
  \rho : &\CH & \to & W \\
  &(y,x) & \mapsto & c_{y,x} \ ,
\end{array}
\]
by the following steps (note that $(y,x)$ lies in $\CH$
if and only if the multiset $\LL(y)$ contains $x$).
\begin{enumerate}[1.]
\item Consider a small loop in $\BC^n - \CH$, which always stays in
  the fiber $\{(y,t), t\in \BC\}$, and which turns once around $x$ (but not
  around any other $x'$ in $\LL(y)$).
\item This loop determines an element $b_{y,x}$ of $\pi_1(\BC^n -
  \CH)=\pi_1(\Vreg / W)$, which is the braid group $B(W)$ of $W$.
\item Send $b_{y,x}$ to $c_{y,x}$ via a fixed surjection
  $B(W) \surj W$.
\end{enumerate}

The map $\rho$ has the following fundamental properties.

\begin{enumerate}[(P1)]
\item\label{item:P1} If $(x_1,\dots, x_p)$ is the ordered support of $\LL(y)$
  (for the lexicographical order on $\BC \simeq \BR^2$), then the
  $p$-tuple $(c_{y,x_1}, \dots, c_{y,x_p})$ lies in $\Fact_p(c)$.
\item\label{item:P2} For all $x \in \LL(y)$, $c_{y,x}$ is a parabolic
  Coxeter element; its length is equal to the multiplicity of $x$ in
  $\LL(y)$, and its conjugacy class corresponds (via the
  bijection of Prop.\ \ref{prop:parab}) to the unique stratum
  $\Lambda$ in $\Lb$ such that $(y,x)\in \Lambda^0$.
\end{enumerate}

According to Property \hyperref[item:P1]{(P1)}, we call the tuple
$(c_{y,x_1}, \dots, c_{y,x_p})$ (where $(x_1,\dots, x_p)$ is the
ordered support of $\LL(y)$) the \emph{factorization of $c$ associated
  to $y$}, and we denote it by $\fact(y)$.

Any block factorization determines a composition of $n$. To any
configuration of $E_n$ we can also associate a composition of $n$,
formed by the multiplicities of its elements in the lexicographical
order. Then Property \hyperref[item:P2]{(P2)} implies that for any $y$
in $Y$, the compositions associated to $\LL(y)$ and $\fact(y)$ are the
same. The third fundamental property (see \cite[Thm.\
5.1]{Ripollfacto} or \cite[Thm.\ 7.9]{BessisKPi1}) is the following.
\begin{enumerate}[(P3)]
\item \label{item:P3} The map $\LL \times \fact : Y \to E_n
  \times \Fact(c)$ is injective, and its image is the entire set of
  compatible pairs (\ie, pairs with same associated composition).
\end{enumerate}

In other words, for each $y\in Y$, the fiber $\LL^{-1}(\LL(y))$ is in
bijection (via $\fact$) with the set of factorizations whose
associated composition of $n$ is the same as that associated to
$\fact(y)$. This fundamental property is a reformulation of a theorem
by Bessis; the proof still relies on some case-by-case analysis.

\section{Lyashko-Looijenga extensions}
\label{sec:LLext}

Property \hyperref[item:P3]{(P3)} is particularly helpful to compute
algebraically certain classes of factorizations. For example, if $y$
lies in $Y-\CK$, then $\fact(y)$ is in $\Fact_n(c)$ (in other words,
it is a reduced decomposition of $c$), \ie, the associated composition
is $(1,1,\dots, 1)$. Thus, from \hyperref[item:P3]{(P3)}, the
set~$\Red_{\CR}(c)$ is in bijection with any generic fiber of $\LL$
(the fiber of any point in $\Enreg$), so it has cardinality $n!\ h^n/\
|W|$, because of Property \hyperref[eq:degLL]{(P0)}. Note that this
number has been computed algebraically, using the fact that the
algebraic morphism $\LL$ is ``weighted-homogeneous''.

In order to go further and count more complicated factorizations of
$c$, we need a more precise algebraic study of the morphism $\LL$, in
particular its restriction to the bifurcation locus $\CK$.

\subsection{\texorpdfstring{Ramification locus for $\LL$}{Ramification locus for LL}} 
~\medskip
\label{subsec:LLram}

Let us first explain the reason why $\LL$ is étale on $Y-\CK$ (as
stated in Property \hyperref[eq:degLL]{(P0)}), where we recall:
\[\CK = \{y\in Y \tq \text{ the multiset } \LL(y) \text{ has multiple
  points} \} \ .\] The argument goes back to Looijenga (in
\cite{looijenga}), and is used without details in the proof of Lemma 5.6 of
\cite{BessisKPi1}.

\medskip

We begin with a more general setting. Let $n\geq 1$, and $P \in
\BC[T_1,\dots, T_n]$ of the form:
\[ P= T_n^n + a_2 (T_1,\dots, T_{n-1}) T_n^{n-2} + \dots +
a_n(T_1,\dots T_{n-1}) \] (here the polynomials $a_i$ do not need to
be quasi-homogeneous). As in the case of $\LL$ we define the
hypersurface $\CH := \{P=0\} \subseteq \BC^n$, and a map $\psi :
\BC^{n-1} \to E_n$, sending ${y=(T_1,\dots, T_{n-1})\in \BC^{n-1}}$ to
the multiset of roots of $P(y,T_n)$ (as a polynomial in $T_n$). This
map can also be considered as the morphism $y\mapsto (a_2(y),\dots,
a_n(y))$.

We set:
\[J_\psi(y)=\Jac((a_2,\dots,a_n) / y)=\det \left(
  \frac{\partial a_i}{\partial T_j} \right)_{\substack{2 \leq i \leq n\\ 1\leq j \leq
  n -1}} \]

\begin{propo}[after Looijenga]
  \label{propgeneralposition}
  With the notations above, let $y$ be a point in $\BC^{n-1}$,
  with~$\psi(y)$ being the multiset $\{x_1,\dots, x_n \}$. Suppose
  that the $x_i$'s are pairwise distinct.

  Then the points $(y,x_i)$ are regular on $\CH$. Moreover, the $n$
  hyperplanes tangent to $\CH$ at $(y,x_1), \dots , (y,x_n)$ are in
  general position if and only if $J_\psi(y)\neq 0$ (\ie, $\psi$ is
  étale at $y$).
\end{propo}

\begin{proof}
  Let $\alpha$ be a point in $\CH$. If it exists, the hyperplane tangent to
  $\CH$ at $\alpha$ is directed by its normal vector: $ \grad_{\alpha}P =
  \left( \frac{\partial P}{\partial T_1}(\alpha), \dots, \frac{\partial
      P}{\partial T_n}(\alpha) \right)$.
 
  Let $y$ be a point in $\BC^{n-1}$ such that the $x_i$'s associated are
  pairwise dictinct. Then the polynomial in $T_n$, $P(y,T_n)$ has the
  $x_i$'s as simple roots, so for each $i$, $\frac{\partial P}{\partial
    T_n}(y,x_i) \neq 0$, and the point $(y,x_i)$ is regular on $\CH$.

  The tangent hyperplanes associated to $y$ are in general position if and
  only if $\det M_y \neq 0$, where $M_y$ is the matrix with columns:
  \[ \left(\grad_{(y,x_1)} P \ ;\ \dots \ ;\ \grad_{(y,x_n)} P\right).\]
  
  After computation, we get: $M_y= A_y V_y$, where

\[
A_y=\left[
  \begin{array}{c|c}
    \begin{array}{c} 0 \\ \vdots \\ 0 \end{array}
    & 
    \displaystyle{\left( \frac{\partial a_j}{\partial T_i} \right)_{\substack{1\leq i
            \leq n-1 \\ 2\leq j \leq n}} }
        \\
        \hline
     n  
     &
     \begin{array}{cccc}  0  &  (n-2)a_2(y)  &  \dots  &  a_{n-1}(y) \end{array}
   \end{array}
 \right]
 \text{\ and\ }
 V_y=\left[
   \begin{array}{ccc}
     x_1^{n-1} & \dots & x_n^{n-1}\\
     \vdots & \ddots & \vdots \\
     x_1 & \dots & x_n\\
     1 & \dots & 1
   \end{array}
 \right] .
\]

  As the $x_i$'s are distinct, the Vandermonde matrix $V_y$ is
  invertible. As $\det A_y = n J_{\psi}(y)$, we can conclude that $\det M_y \neq
  0$ if and only if $J_\psi(y)\neq 0$.
\end{proof}

If the $x_i$'s are not distinct, nothing can be said in general. But if
$\psi$ is a Lyashko-Looijenga morphism $\LL$, then we can deduce the
following property.

\begin{coro}
  \label{coroetale}
  Let $y$ be a point in $\BC^{n-1}$, and suppose that $\LL(y)$ contains $n$
  distinct points. Then $J_{\LL}(y)\neq 0$.

  In other words, $\LL$ is étale on (at least) $Y-\CK$.
\end{coro}

\begin{proof}
  Set $\LL(y)=\{x_1,\dots, x_n\}$. As the $x_i$'s are distinct, from Lemma
  \ref{propgeneralposition} one has to study the hyperplanes tangent to
  $\CH$ at $(y,x_1), \dots , (y,x_n)$. By using their characterization in
  terms of basic derivations of $W$, it is straightforward to show that the
  $n$ hyperplanes are always in general position: we refer to the proof of
  \cite[Lemma 5.6]{BessisKPi1}.
\end{proof} 

In the following we will prove the equality $Z(J_{\LL})=\CK$, \ie, that $\LL$ is
étale exactly on~${Y-\CK}$.

\subsection{\texorpdfstring{The well-ramified property for $\LL$}{The well-ramified property for LL}}
\label{subsec:LLwell}
~\medskip

Following Remark \ref{rk:algLL}, consider $\LL$ as the algebraic
morphism
\[\begin{array}{ccc}
  \BC^{n-1} & \to & \BC^{n-1}\\
  (f_1,\dots, f_{n-1}) & \mapsto & (a_2(f_1,\dots, f_{n-1})\ ,\
  \dots\ ,\  a_n(f_1, \dots, f_{n-1}))\ .
\end{array} \]

According to \cite[Thm.\ 5.3]{BessisKPi1}, this is a \emph{finite}
quasihomogeneous map (for the weights $\deg(f_i)=d_i$, $\deg
a_j=jh$). So we get a graded finite polynomial extension
\[ A=\BC[a_2,\dots,a_n] \subseteq \BC[f_1,\dots, f_{n-1}]=B.\]

Such extensions are studied in \cite{vrg}. Let us recall the
properties and definitions that we need. For such an extension
$A\subseteq B$, we denote by $\Spec_1(B)$ the set of ideals of $B$ of
height one, and $\Spram(B)$ its subset consisting of ideals which are
ramified over $A$. These ideals are principal, and we will talk about
``the set of ramified polynomials of the extension'' for a set of
representatives of generators of these ramified ideals.

In \cite[Thm.\ 1.8]{vrg} we described the factorization of the
Jacobian polynomial of the extension $J_{B/A}$ . We can apply it here
and obtain:
\begin{equation} \tag{*} \label{eq:jac} J_{\LL}=\det \left( \frac{\partial a_i}{\partial f_j}
\right)_{\substack{2 \leq i \leq n\\1\leq j \leq n-1}}\doteq \prod_{Q \in
  \Spram(B)} Q ^{e_Q -1} \ , 
\end{equation}
where $e_Q$ is the ramification index of $Q$ (and $\doteq$ designates
equality up to a scalar).

We also introduced in \cite{vrg} the notion af a well-ramified extension:

\begin{defi}\label{defwellram}
  A finite graded polynomial extension $A \subseteq B$ is
  \emph{well-ramified} if:
\[ \left( J_{B/A} \right) \cap A = \bigg( \prod_{Q \in \Spram(B)} Q ^{e_Q}
\bigg) \ \text{as an ideal of } A. \]
\end{defi}

Well-ramified extensions are generalizations of Galois extensions
(where $A$ is the algebra of invariants of $B$ under the action of a
reflection group), but keep some of their characteristics. We refer to
\cite[Sec.\ 3.2]{vrg} for details and other characterizations of this
property. The name ``well-ramified'' is chosen accordingly to one of
these characterizations, namely :
\begin{enumerate}[\hspace{0.8cm}]
\item ``For any $\fp \in \Spec_1(A)$, if there exists $\fq_0 \in
  \Spec_1(B)$ over $\fp$ which is ramified, then any other $\fq \in
  \Spec_1(B)$ over $\fp$ is also ramified.'' (\cite[Prop.\ 3.2.(iv)]{vrg})
\end{enumerate}

In the following of this subsection we prove that the extension
defined by $\LL$ is well-ramified. In Sec.\ \ref{subsec:LLvrg} we
will compare the setting of Lyashko-Looijanga extensions to that of
Galois extensions.

\bigskip

We recall from Sec.\ \ref{subsec:LL} the definition of $D_{\LL}$:
\[ D_{\LL}:= \Disc(f_n^n + a_2 f_n^{n-2} + \dots + a_n\ ;\ f_n)\ , \] so
that $\CK=\LL^{-1}(E_n - \Enreg)$ is the zero locus of $D_{\LL}$ in
$Y$. We denote by $\Lb_2$ the set of all closed strata in $\Lb$ of
codimension $2$. Note that $\Lb_2$ is also the set of conjugacy
classes of elements of $\NC(W)$ of length $2$ (cf. Prop.\ \ref{prop:parab}).

We define the following map
\[\begin{array}{lccc}
\varphi :& W \qg V \simeq Y \times \BC & \to & Y\\
& \bar{v}=(y,x) & \mapsto & y \ .
\end{array}\]

Then, using the notations and properties of Sec.\ \ref{subsec:lbl}, we
have:
\[\begin{array}{lll}
y \in \CK & \Leftrightarrow & \exists x \in \LL(y), \text{ with
  multiplicity} \geq 2 \\
& \Leftrightarrow & \exists x \in \LL(y), \text{ such that } \ell(c_{y,x})\geq 2\\
& \Leftrightarrow & \exists x \in \LL(y), \text{ such that } (y,x)\in \Gamma^0
\text{ for some stratum } \Gamma \in \Lb \text{ of codim.} \geq 2\\
 & \Leftrightarrow & \exists x \in \LL(y),\ \exists \Lambda \in \Lb_2 ,
 \text{ such that } (y,x)\in \Lambda\\
& \Leftrightarrow & \exists \Lambda \in \Lb_2 , \text{ such that } y\in
\varphi(\Lambda).
\end{array} \]

So the hypersurface $\CK$ is the union of the $\varphi(\Lambda)$, for
$\Lambda \in \Lb_2$. It can be shown that they are in fact its
irreducible components (cf. \cite[Prop.\ 7.4]{Ripollfacto}). Thus we
can write
\begin{equation} \tag{**} \label{eq:disc} D_{\LL}= \prod_{\Lambda \in
    \Lb_2} D_{\Lambda}^{r_{\Lambda}},
\end{equation}
for some $r_\Lambda \geq 1$, where the $D_{\Lambda}$ are irreducible
(homogeneous) polynomials in $B=\BC[f_1,\dots, f_{n-1}]$ such that
$\varphi(\Lambda)=\{D_{\Lambda}=0 \}$.

\begin{example}
\label{ex:A3_2}
In the example of $A_3$ in Fig.\ \ref{A3}, the two strata of $\Lb_2$
described in Ex.\ \ref{ex:A3} ---let us call them
$\Lambda_{\text{red}}$ and $\Lambda_{\text{blue}}$--- project (by
$\varphi$) onto the two irreducible components of $\CK$. The explicit
computation gives that the power $r_{\Lambda_{\text{red}}}$ of
$D_{\Lambda_{\text{red}}}$ (resp. for blue) in $D_{\LL}$ equals $2$
(resp.~$3$), which is also the common order of parabolic Coxeter
elements in the conjugacy class corresponding to the strata (indeed,
those are products of two commuting transpositions,
resp. $3$-cycles). This turns out to be a general phenomenon, as
described in the following theorem.
\end{example}

Now we give an important interpretation of the integers
$r_\Lambda$, and deduce that $\LL$ is a well-ramified extension.

\begin{theo}
  \label{thm:LL}
  Let $\LL$ be the Lyashko-Looijenga extension associated to a
  well-generated, irreducible complex reflection group, together with
  the above notations. For any $\Lambda$ in $\Lb_2$, let~$w$ be a
  (length $2$) parabolic Coxeter element of $W$ in the conjugacy class
  corresponding to $\Lambda$. Recall that $r_\Lambda$ denotes the power of
  $D_{\Lambda}$ in $D_{\LL}$. Then:

  \begin{enumerate}[(a)]
  \item The integer $r_{\Lambda}$ is the number of reduced
    decompositions of $w$ into two reflections. When~$W$
    is a $2$-reflection group\footnote{A $2$-reflection group is a
      complex reflection group generated by reflections of order $2$;
      see \cite[Thm.\ 2.2]{BessisKPi1} for an interesting property of
      these groups.},
    it is simply the order of $w$.
  \item The set of ramified polynomials of the extension $A \subseteq
    B$ is the family $\{D_\Lambda \tq \Lambda \in \Lb_2 \}$, and the
    ramification index of $D_{\Lambda}$ is $r_{\Lambda}$.
  \item The $\LL$-Jacobian satisfies: $\displaystyle{J_{\LL}\doteq
      \prod_{\Lambda\in \Lb_2 } D_{\Lambda} ^{r_{\Lambda} -1}}$.
  \item The ``$\LL$-discriminant'' $\displaystyle{D_{\LL}=
      \prod_{\Lambda\in \Lb_2 } D_{\Lambda} ^{r_{\Lambda}} }$ is a
    generator for the ideal $(J_{\LL})\cap A$.
  \item The polynomial extension associated to $\LL$ is well-ramified.
  \end{enumerate}
\end{theo}

\begin{proof}
  Let us prove first that all the ramified polynomials in $B$ are
  included in ${\{D_{\Lambda} \tq \Lambda \in \Lb_2 \}}$.  The
  polynomial $D_{\LL}$ is irreducible in $A$ since, as a polynomial in
  $a_2,\dots, a_n$, it is the discriminant of a reflection group of
  type $A_{n-1}$. Therefore, for all $\Lambda$ in $\Lb_2$, the
  inclusion
  \[ (D_{\Lambda})\cap A \supseteq (D) \] is an inclusion between
  prime ideals of height one in $A$. So we have $(D_{\Lambda})\cap A
  =(D)$, and the ramification index $e_{D_\Lambda}$ is equal to
  $v_{D_\Lambda}(D)=r_\Lambda$.  According to Corollary
  \ref{coroetale}, if $J_{\LL}(y)=0$, then $\LL(y)\notin \Enreg$. So
  the variety of zeros of $J_{\LL}$ (defined by the ramified
  polynomials in $B$) is included in the preimage \[\LL^{-1}(Z(D)) =
  \bigcup_{\Lambda \in \Lb_2} Z(D_{\Lambda})\ .\] Thus, any ramified
  polynomial of the extension is necessarily one of the
  $D_{\Lambda}$'s.

  \medskip

  Let us prove the point \emph{(a)}. Let $\Lambda \in \Lb_2$, and
  $\mu$ be the composition $(2,1,\dots, 1)$ of $n$. Choose
  $\xi=(w,s_3,\dots, s_n)$ in $\Fact_\mu (c)$ such that the conjugacy
  class of $w$ (the only element of length~$2$ in $\xi$) corresponds
  to $\Lambda$. Fix $e\in E_n$, with composition type $\mu$, and such
  that the real parts of its support are distinct. There exists a
  unique $y_0$ in $Y$, such that $\LL(y_0)=e$ and $\fact(y_0)=\xi$ (by
  Property \hyperref[item:P3]{(P3)} in Sec.\
  \ref{subsec:lbl}). Moreover $y_0$ lies in $\varphi(\Lambda)$
  (Property \hyperref[item:P2]{(P2)}). Using the precise definition of
  the map $\fact$ \cite[Def.\ 4.2]{Ripollfacto}, and the ``Hurwitz
  rule'' \cite[Lemma 4.5]{Ripollfacto}, we deduce that for a
  sufficiently small connected neighbourhood $\Omega_0$ of $y_0$, if
  $y$ is in $\Omega_0 \cap (Y- \CK)$, then $\fact(y)$ is in
  \[ F_w:= \{(s_1',s_2',\dots, s_n ')\in \Red_\CR (c) \tq s_1's_2'=w
  \text{ and } s_i'=s_i \ \forall i\geq 3 \} .\] Let us fix $y$ in
  $\Omega_0 \cap (Y- \CK)$. Then, because of Property
  \hyperref[item:P3]{(P3)}, we get an injection
  \[ \fact : \LL^{-1}(\LL(y)) \cap \Omega_0 \ \inj \ F_w \ .\] But
  this map is also surjective, thanks to the covering properties of
  $\LL$ and the transitivity of the Hurwitz action on $w$.
  Indeed, we can ``braid'' $s_1'$ and $s_2'$ (by cyclically
  intertwining the two corresponding points of $\LL(y)$, while staying
  in the neighbourhood) so as to obtain any factorization of
  $w$. Thus:
  \[|\LL^{-1}(\LL(y))\cap \Omega_0| = |F_w|\ .\]

  Using the classical characterization of the ramification index (see
  \emph{e.g.} \cite[Prop.\ 2.4]{vrg}), we infer that $|F_w|$ is equal
  to the ramification index $e_{D_{\Lambda}}$, so: $ r_\Lambda =
  |F_w|$. This is also the number of reduced decompositions of $w$,
  \ie, the Lyashko-Looijenga number for the parabolic subgroups in the
  conjugacy class $\Lambda$.

  For any rank $2$ parabolic subgroup with degrees $d_1',h'$, the
  $\LL$-number is $2h'/d_1'$. In the particular case when $W$ is a
  $2$-reflection group, such a subgroup is a dihedral group, hence
  $d_1'$ equals $2$ and $r_\Lambda$ is the order $h'$ of the
  associated parabolic Coxeter element $w$.

  \medskip 

  Consequently, for all $\Lambda \in \Lb_2$,
  $e_{D_\Lambda}=r_{\Lambda}$ is strictly greater than $1$, so
  $D_{\Lambda}$ is ramified, and statement \emph{(b)} is proven. Using
  Formula (\ref{eq:jac}) above, this also directly implies \emph{(c)}.

  Moreover, we obtain:
  \[ \prod_{Q \in \Spram(B)} Q ^{e_Q} = \prod_{\Lambda \in \Lb_2}
  D_\Lambda ^{e_{D_\Lambda}} = D_{\LL} \ , \] so this polynomial lies
  in $A$.  We recognize one of the characterizations of a
  well-ramified extension (namely \cite[Prop.\ 3.2.(iii)]{vrg}), from
  which we deduce \emph{(d)} and \emph{(e)}.
\end{proof}

\subsection{A more intrinsic definition of the Lyashko-Looijenga Jacobian}
~\medskip

\label{subsec:saito}

In this subsection we give an alternate definition for the Jacobian
$J_{\LL}$, which is more intrinsic, and which allows to recover a formula
observed by K. Saito.

We will use the following elementary property. Suppose $P\in
\BC[T_1,\dots, T_{n-1},X]$ has the form:
\[ P = X^n + b_1 X^{n-1} + \dots + b_n \ , \]
with $b_1,\dots, b_n \in \BC[T_1,\dots, T_{n-1}]$. Note that we do not
require $b_1$ to be zero. Let us denote by $J(P)$ the polynomial:
\[ J(P):= \Jac \left( \left( P, \DP{P}{X}, \dots, \DP[n-1]{P}{X} \right)
  \middle/  (T_1,\dots, T_{n-1},X) \right) \ . \]

\begin{lemma}
  Let $P$ be as above. We set $Y=X+\frac{b_1}{n}$ and denote by $Q$ the
  polynomial in $\BC[T_1,\dots, T_{n-1},Y]$ such that $Q(T_1,\dots,
  T_{n-1},Y)=P(T_1,\dots, T_{n-1},X)$, so that
  $ Q = Y^n + a_2 Y^{n-2} + \dots + a_n$,  with $a_2,\dots, a_n \in
  \BC[T_1,\dots, T_{n-1}]$. 

  We define $J(P)$ as above and $J(Q)$ similarly ($Y$ replacing $X$). Then:
  \begin{enumerate}[(i)]
  \item $J(P)=J(Q)$;
  \item $J(P)$ does not depend on $X$, and $J(P)\doteq \Jac((a_2,\dots, a_n) /
    (T_1,\dots, T_{n-1}))$.
  \end{enumerate}
\end{lemma}

The proof is elementary, and can be found in \cite[Lemma
3.4]{These}. Consequently, we have an intrinsic definition for the
Lyashko-Looijenga Jacobian:
\[ J_{\LL} \doteq J(\Delta_W)=\Jac\left( \left( \Delta_W,
    \DP{\Delta_W}{f_n}, \dots, \DP[n-1]{\Delta_W}{f_n} \right)
  \middle/ (f_1,\dots, f_{n}) \right) \ . \] where $f_1,\dots, f_n$ do
not need to be chosen such that the coefficient of $f_n^{n-1}$ in
$\Delta_W$ is zero. Note that for the computation of $D_{\LL}$ as well,
the fact that the coefficient $a_1$ is zero in $\Delta_W$ is not
important, because of invariance by translation.

\medskip

With these alternative definitions, the factorization of the Jacobian
given by Thm.\ \ref{thm:LL} has already been observed (for real
groups) by Kyoji Saito: it is Formula 2.2.3 in
\cite{saitopolyhedra}. He uses this formula in his study of the
semi-algebraic geometry of the quotient $W \qg V$.
His proof was case-by-case and detailed in an unpublished extended
version of the paper \cite{saitopolyhedra} (\cite[Lemma
3.5]{saitopolyhedra2}).

\subsection{The Lyashko-Looijenga extension as a virtual reflection group}
\label{subsec:LLvrg}
~\medskip

\begin{table}[!h]
{\renewcommand{\arraystretch}{1.8}
\begin{tabular}{@{\vrule width 1pt\,}m{2.4cm}@{\,\vrule width 1pt\,}>{\centering}m{6.2cm}@{\,\vrule width 1pt\,}>{\centering}m{6.7cm}@{\,\vrule width 1pt}|}
  \hlinewd{1pt}

  &Complex reflection group & Lyashko-Looijenga extension
  \tabularnewline \hlinewd{1pt}
  Morphism: & 
$\begin{array}{rcl}
    p :\quad  V & \to & W \qg V \\
    (v_1,\dots, v_n)  & \mapsto & (f_1(v),\dots, f_n(v))
  \end{array}$
  &
$\begin{array}{rcl}
    \LL : \quad Y & \to &\BC^{n-1} \\
     (y_1,\dots,y_{n-1}) & \mapsto & (a_2(y),\dots, a_n(y))
   \end{array}$
   \tabularnewline\hline
   Weights: & $\deg v_j=1$ ; $\deg f_i=d_i$ & $\deg y_j=d_j$ ; $\deg a_i=ih$\tabularnewline\hline
   Extension: & $\BC[f_1,\dots, f_n]=\BC[V]^W \subseteq \BC[V]$ & 
   $\BC[a_2,\dots, a_n] \subseteq \BC[y_1, \dots, y_{n-1}]$ \tabularnewline\hline
   Free, of rank: & $|W|=d_1\dots d_n$\ ;\  Galois&
   $n!h^n / |W|=\prod ih / \prod d_j$; non-Galois\tabularnewline\hline
  
  Unramified covering: & $\Vreg \ \surj \ W \qg \Vreg$ & $Y-\CK \ \surj \ \Enreg$\tabularnewline\hline
  Generic fiber: & $\simeq W$ & $\simeq  \Red_\CR(c)$\tabularnewline\hline
  Ramified part: & $\bigcup_{H \in \CA} H \ \surj \ (\bigcup H) / W=\CH$ & $\CK=
  \bigcup_{\Lambda \in \Lb_2} \varphi(\Lambda) \ \surj \ E_{\alpha}$
  \tabularnewline\hline
  Discriminant: & $\Delta_W= \prod_{H\in \CA} \alpha_H^{e_H}\in \BC[f_1,\dots, f_n]$ &
  $D_{\LL}=\prod_{\Lambda \in \Lb_2} D_{\Lambda}^{r_{\Lambda}} \in
  \BC[a_2,\dots, a_n]$\tabularnewline\hline
  Ramification indices: & $e_H=|W_H|$ & $r_{\Lambda}=$ order of parabolic elements of type
  $\Lambda$ \tabularnewline\hline
  Jacobian: & $J_W= \prod \alpha_H^{e_H-1}\in
    \BC[V]$& $J_{\LL}=\prod
    D_{\Lambda}^{r_{\Lambda}-1} \in \BC[f_1,\dots,
    f_{n-1}]$\tabularnewline   \hlinewd{1pt}
\end{tabular}}
\caption{Analogies between Galois extensions and Lyashko-Looijenga extensions.}
\label{analogies}
\end{table}

In \cite{vrg} we discussed some properties of well-ramified
extensions, and explained that they can be regarded as an analogous of
the invariant theory of reflection groups. Indeed, considering a finite
graded polynomial extension $A\subseteq B$, if the polynomial algebra
$A$ is the invariant algebra $B^W$ of $B$ under a group action, then
$W$ is a complex reflection group (by Chevalley-Shephard-Todd's
theorem). Here, for $\LL$ extensions, the situation is similar, but
$A$ is \emph{not} the invariant ring of $B$ under some group
action. Still, many properties remain valid. Following Bessis, we use
the term \emph{virtual reflection group} for this kind of
extensions. The general situation is discussed in \cite{vrg}.

In Table \ref{analogies} we list the first analogies between the
setting of a Galois extension (polynomial extension with a reflection
group acting) and that of a Lyashko-Looijenga extension regarded as a
virtual reflection group. This is not an exhaustive list, and we may
wonder if the analogies can be made further.

\section{Combinatorics of the submaximal factorizations}
\label{sec:submax}

In this section we are going to use properties of the morphism $\LL$
to count specific factorizations of a Coxeter element; this will lead,
thanks to Thm.\ \ref{thm:LL}, to a geometric proof of a particular
instantiation of Chapoton's formula.

\medskip

We call \emph{submaximal factorization} of a Coxeter element $c$ a
block factorization of $c$ with $n-1$ factors, according to Def.\
\ref{def:facto}. Thus, submaximal factorizations contains $(n-2)$
reflections and one factor of length $2$, and are a natural first
generalization of the set of reduced decompositions~$\Red_\CR
(c)$. These are included in the more general ``primitive''
factorizations studied in \cite{Ripollfacto}.

\bigskip

\subsection{Submaximal factorizations of type $\Lambda$}
\label{subsec:submaxtype}
~\medskip

Let $\Lambda$ be a stratum of $\Lb_2$: it corresponds
(cf. Prop.\ \ref{prop:parab}), to a conjugacy class of parabolic Coxeter
elements of length $2$. We say that a submaximal factorization is
\emph{of type $\Lambda$} if its factor of length $2$ lies in this
conjugacy class. We denote by $\Fact_{n-1}^{\Lambda}(c)$ the set of
such factorizations. Using the relations between $\LL$ and $\fact$, we
can count these factorizations.

\medskip

For $\Lambda$ a stratum of $\Lb_2$, let us define the following restriction
of $\LL$:
\[ \LL_{\Lambda} \ : \ \varphi(\Lambda) \to E_\alpha \ ,\] where
$E_\alpha= E_n - \Enreg$. We denote by $E_\alpha ^0$ the
subset of $E_\alpha$ constituted by the configurations whose partition
(of multiplicities) is exactly $\alpha=2^1 1^{n-2}$.

We define $\varphi(\Lambda)^0=\LL_{\Lambda}^{-1}(E_\alpha ^0)$, and
$\CK^0=\LL^{-1}(E_\alpha ^0)=\cup_{\Lambda \in \Lb_2}
\varphi(\Lambda)^0$. We recall from \cite{Ripollfacto} the
following properties:
\begin{itemize}
\item the restriction of $\LL$ : $\CK^0 \ \surj\ E_\alpha ^0$ is a
  (possibly not connected) unramified covering \cite[Thm.\ 5.2]{Ripollfacto};
\item the connected components of $\CK^0$ are the $\varphi(\Lambda)^0$, for
  $\Lambda \in \Lb_2$;
\item the image, by the map $\fact$, of $\varphi(\Lambda)^0$ is exactly
  $\Fact_{n-1}^{\Lambda}(c)$;
\end{itemize}

\medskip

The map $\LL_\Lambda$ defined above is an algebraic morphism, corresponding
to the extension
\[ \BC[a_2,\dots, a_n] / (D) \subseteq \BC[f_1,\dots, f_{n-1}]/(D_\Lambda)
\ .\]

\begin{theo}
  \label{thm:fact}
  Let $\Lambda$ be a stratum of $\Lb_2$. Then:
  \begin{enumerate}[(a)]
  \item $\LL_\Lambda$ is a finite quasi-homogeneous morphism of degree
    $\frac{(n-2)!\ h^{n-1}}{|W|} \deg D_\Lambda $;
  \item the number of submaximal factorizations of $c$ of type $\Lambda$ is
    equal to
    \[ |\Fact_{n-1}^{\Lambda}(c)| = \frac{(n-1)!\ h^{n-1}}{|W|} \deg
    D_\Lambda \ .\]
  \end{enumerate}
  \end{theo}

\begin{proof}
  From Hilbert series, we get that $\LL_\Lambda$ is a finite free extension
  of degree
    \[ \left. \frac{\prod \deg (a_i)}{\deg(D)}\  \middle/ \ \frac{\prod
    \deg(f_i)}{\deg(D_\Lambda)} \right. = \frac{n!\ h^n}{|W|}\frac{\deg
    D_\Lambda}{\deg D} .\]
  \emph{(a)} As $D_{\LL}$ is a discriminant of type $A$ for the variables $a_2,\dots, a_n$ of
    weights $2h,\dots, nh$, we have $\deg D_{\LL}= n(n-1)h$. Thus:
    \[ \deg(\LL_\Lambda) = \frac{(n-2)!\ h^{n-1}}{|W|} \deg D_\Lambda
    .\] \emph{(b)} This degree is also the cardinality of a generic
    fiber of $\LL_\Lambda$, \ie, $|\LL^{-1}(\eps) \ \cap \
    \varphi(\Lambda)|$, for~${\eps \in E_\alpha^0}$. Consequently, from
    Property \hyperref[item:P3]{(P3)} in Sec.\ \ref{subsec:lbl}, it
    counts the number of submaximal factorizations of type $\Lambda$,
    where the length $2$ element has a \emph{fixed} position (given by
    the composition of $n$ associated to $\eps$). There are $(n-1)$
    compositions of partition type $\alpha\vdash n$, so we obtain
    $|\Fact_{n-1}^{\Lambda}(c)|=(n-1)\deg (\LL_\Lambda) =
    \frac{(n-1)!\ h^{n-1}}{|W|} \deg D_\Lambda$.
  \end{proof}

\begin{remark}
\label{rkconcat}
Let us denote by $\Fact_{(2,1,\dots, 1)}^{\Lambda}(c)$ the set of submaximal
factorizations of type~$\Lambda$ where the length $2$ factor is in first
position. By symmetry, formula (b) is equivalent to
\[ |\Fact_{(2,1,\dots, 1)}^{\Lambda}(c)|=\frac{(n-2)!\ h^{n-1}}{|W|} \deg
D_\Lambda \ .\] As $\sum r_{\Lambda} \deg D_{\Lambda}= \deg D_{\LL}=n(n-1)h$,
this implies the equality :
\[ \sum_{\Lambda \in \Lb_2} r_{\Lambda} |\Fact_{(2,1,\dots,
  1)}^{\Lambda}(c)|= \frac{(n-2)!\ h^{n-1}}{|W|} \deg D_{\LL}=
\frac{n!h^n}{|W|}=|\Red_\CR(c)| \ .\] This formula reflects a property
of the following concatenation map:
\[ \begin{array}{ccc}
\Red_\CR (c) & \surj & \Fact_{(2,1,\dots, 1)}(c)\\
(s_1,s_2,\dots, s_n) & \mapsto & (s_1 s_2, s_3,\dots, s_n) \ ,
\end{array}
\]
namely, that the fiber of a factorization of type $\Lambda$ has
cardinality $r_{\Lambda}$ (which is the number of factorizations of
the first factor in two reflections).
\end{remark}

\begin{remark}
\label{rkkra}
  In \cite{KraMu}, motivated by the enumerative theory of the generalized
  noncrossing partitions, Krattenthaler and Müller defined and computed
  the \emph{decomposition numbers} of a Coxeter element, for all
  irreducible \emph{real} reflection groups. In our terminology, these are
  the numbers of block factorizations according to the Coxeter type of the
  factors. Note that the Coxeter type of a parabolic Coxeter element is the
  type of its associated parabolic subgroup, in the sense of the
  classification of finite Coxeter groups. So the conjugacy class for a
  parabolic elements is a finer characteristic than the Coxeter type: take
  for example $D_4$, where there are three conjugacy classes of parabolic
  elements of type $A_1\times A_1$.

  Nevertheless, when $W$ is real, most of the results obtained from formula
  (b) in Thm.\ \ref{thm:fact} are very specific cases of the computations in
  \cite{KraMu}. But the method of proof is completely different, geometric
  instead of combinatorial\footnote{The computation of all decomposition
    numbers for complex groups, by combinatorial means, is also a work in
    progress (Krattenthaler, personal communication).}. Note that another possible
  way to tackle this problem is to use a recursion, to obtain data for the
  group from the data for its parabolic subgroups. A recursion formula (for
  factorizations where the rank of each factor is dictated) is indeed given
  by Reading in \cite{reading}, but the proof is very specific to the real
  case.
\end{remark}

For $W$ non-real, formula (b) implies new combinatorial results on the
factorization of a Coxeter element. The numerical data for all
irreducible well-generated complex reflection groups are listed in
Section \ref{sec:data}. In particular, we obtain (geometrically)
general formulas for the submaximal factorizations of a given type in
$G(e,e,n)$.

\subsection{Enumeration of submaximal factorizations of a Coxeter element}
\label{subsec:submax}
~\medskip

Thanks to Thms~\ref{thm:LL} and \ref{thm:fact}, we can now obtain a
formula for the number of submaximal factorizations, with a
geometric proof:

\begin{coro}
\label{cor:submax}
  Let $W$ be an irreducible well-generated complex reflection group, with
  invariant degrees $d_1\leq\dots \leq d_n=h$. Then, the number of
  submaximal factorizations of a Coxeter element $c$ is equal to:
    \[ |\Fact_{n-1}(c)| = \frac{(n-1)!\ h^{n-1}}{|W|}
\left( \frac{(n-1)(n-2)}{2}h + \sum_{i=1}^{n-1} d_i \right) .\] 
\end{coro}

\begin{proof}
  Using Thm.\ \ref{thm:fact}(b) and Thm.\ \ref{thm:LL}(b)-(c), we compute:
  \[\begin{array}{lcl}
    |\Fact_{n-1}(c)| =|\Fact_{\alpha}(c)| & = &
    \displaystyle{\sum_{\Lambda \in \Lb_2}
      |\Fact_{n-1}^{\Lambda}(c)|}\\ & = &
    \displaystyle{\frac{(n-1)!\ h^{n-1}}{|W|} \sum_{\Lambda \in \Lb_2}
      \deg D_\Lambda }\\ & = &
    \displaystyle{\frac{(n-1)!\ h^{n-1}}{|W|} \left( \deg D_{\LL} - \deg
        J_{\LL} \right) }\ ,
  \end{array} \]
  As $\deg D_{\LL} = n(n-1)h$ and $\deg
  J_{\LL} = \sum_{i=2}^n \deg(a_i) - \sum_{j=1}^{n-1}
  \deg(f_j)=\sum_{i=2}^n ih - \sum_{j=1}^{n-1} d_j$, a quick
  computation gives the conclusion.
\end{proof}

\begin{remark}
  The formula in the above theorem is actually included in Chapoton's
  formula: indeed, there exist easy combinatorial tricks allowing to
  pass from the numbers of multichains to the numbers of strict chains
  (which are roughly the numbers of block factorizations). We refer to
  \cite[App.\ B]{These} for details of these relations and general
  formulas for the number of block factorizations predicted by
  Chapoton's formula.

  \emph{However}, the proof we obtained here is more satisfactory (and
  more enlightening) than the one using Chapoton's formula. Indeed, if
  we sum up the ingredients of the proof, we only made use of
  the formula for the Lyashko-Looijenga number $n!\ h^n/\ |W|$ ---
  necessary to prove the first properties of $\LL$ in
  \cite{BessisKPi1} ---, the remaining being the geometric properties
  of $\LL$, for which we never used the classification. In other
  words, we travelled from the numerology of $\Red_\CR(c)$ to that of
  $\Fact_{n-1}(c)$, without adding any case-by-case analysis to the
  setting of \cite{BessisKPi1}.
\end{remark}

 Although it seems to be a new interesting avenue towards a geometric
 explanation of Chapoton's formula, the method used here to compute
 the number of submaximal factorizations is not directly generalizable
 to factorizations with fewer blocks. A more promising approach would
 be to avoid computing explicitely these factorizations, and to try to
 understand globally Chapoton's formula as some ramification formula
 for the morphism $\LL$. A reformulation of the formula gives indeed:
 \[ \forall p \in \BN,\ \sum_{k = 1}^n \binom{p+1}{k} |\Fact_k(c)| =
 \prod_{i=1}^n \frac{d_i+ph}{d_i}\ ,\] where the $\Fact_k$ are closely
 related to the cardinalities of the fibers of $\LL$.

\section{Numerical data for the factorizations of the
  Lyashko-Looijenga discriminants}

\label{sec:data}
\begin{table}[!h]
\rotatebox{90}{
{\small
$   
{\renewcommand{\arraystretch}{1.3}
\begin{disarray}{@{\vrule width 1pt\,}c@{\,\vrule width 1pt\,}c@{\,\vrule width 1pt\,}c@{\,\vrule width 1pt}}
  \hlinewd{1pt}
\text{\begin{tabular}{c} Group type \\ {[Isodiscriminantal groups]}
    \end{tabular}} &
  \slfrac{(n-2)!\ h^{n-1}}{|W|} &
  \LL\text{-data} 
  \\
  \hlinewd{1pt} 
  \begin{array}{c} A_n \ ,\ n \geq 2. \\
    \left[G_4,G_8,G_{16},G_{25},G_{32}\right] \end{array} &
  \slfrac{(n+1)^{n-2}}{(n(n-1))} & \boxed{2} \centerdot \left(\slfrac{n(n-1)(n-2)}{2}\right) +
  \boxed{3} \centerdot \left(n(n-1)\right) \\
  \hline
  \begin{array}{c} B_n \ , \ {n\geq 2}. \\ \left[G(d,1,n),G_5,G_{10},G_{18},G_{26}\right]
    \end{array} &
  \slfrac{n^{n-2}}{(2(n-1))} & 
  \begin{array}{r}
    \boxed{2} \centerdot \left((n-1)(n-2)(n-3)\right) + \boxed{2}
    \centerdot \left(2(n-1)(n-2)\right)  \qquad \qquad \\
    +\: \boxed{3} \centerdot
    \left(2(n-1)(n-2)\right)+ \boxed{4}
    \centerdot \left(2(n-1)\right)
  \end{array}
  \\
  \hline
  \begin{array}{c} 
    I_2(e)\\
    \left[G_6,G_9,G_{17},G_{14},G_{20},G_{21}\right] 
  \end{array}  &
  \slfrac12 & 
  \boxed{e}\centerdot \left(2\right) 
  \\
  \hline
  \begin{array}{c}
    G(e,e,n), \ e \geq 2, n\geq 5 \\
    (=D_n \text{ for } e=2)
  \end{array} &
  \slfrac{(n-1)^{n-2}}n & 
  \boxed{2} \centerdot \left(\slfrac{n(n-2)(n-3)e}{2}\right) +
  \boxed{3} \centerdot \left(n(n-2)e \right) + \boxed{e} \centerdot \left(n\right) 
  \\
  \hline
  G(e,e,3)\ ,\ e\geq 3&
  \slfrac23 & 
  \begin{array}{lcccc}
    \text{If }3\nmid e\ &: &  \boxed{3} \centerdot \left(3e\right) &+& \boxed{e}\centerdot \left(3\right)\\
    \text{If }3\mid e\ &: &  \boxed{3} \centerdot \left(e\right)+ \boxed{3} \centerdot \left(e\right)+
    \boxed{3} \centerdot \left(e\right) &+&\boxed{e}\centerdot \left(3\right) 
  \end{array}
  \\
  \hline
  \begin{array}{c}
    G(e,e,4), \ e \geq 2 \\
    (=D_4 \text{ for } e=2)
  \end{array} &
  \slfrac{9}4 & 
  \begin{array}{lcccc}
    \text{If }e\text{ odd } &: & \boxed{2}\centerdot \left(4e\right) &+& \boxed{3}
    \centerdot \left(8e\right) + \boxed{e} \centerdot \left(4\right)   \\
    \text{If }e\text{ even }& : & \boxed{2}\centerdot \left(2e\right) + \boxed{2}\centerdot \left(2e\right) &+ &\boxed{3}
    \centerdot \left(8e\right) + \boxed{e} \centerdot \left(4\right) 
  \end{array}
  \\
  \hline G_{23} \ (=H_3)& \slfrac56 & \boxed{2} \centerdot \left(6\right) +
  \boxed{3} \centerdot \left( 6 \right) + \boxed{5} \centerdot \left( 6
  \right)
  \\
  \hline G_{24} & \slfrac7{12} & \boxed{3} \centerdot \left(12 \right) +
  \boxed{4} \centerdot \left(12 \right)
  \\
  \hline G_{27} & \slfrac5{12} & \boxed{3} \centerdot \left(12 \right) +
  \boxed{3} \centerdot \left(12 \right) + \boxed{4} \centerdot \left(12
  \right) + \boxed{5} \centerdot \left(12 \right)
  \\
  \hline G_{28} \ (=F_4) & 3 &  \boxed{2} \centerdot \left( 24\right) +
  \boxed{3} \centerdot \left( 8 \right) + \boxed{3} \centerdot \left( 8
  \right) +  \boxed{4} \centerdot \left( 12 \right)
  \\
  \hline G_{29} & \slfrac{25}{12} & \boxed{2} \centerdot \left(24 \right) +
  \boxed{3} \centerdot \left(48 \right) + \boxed{4} \centerdot \left(12
  \right)
  \\
  \hline G_{30} \ (=H_4) & \slfrac{15}4 & \boxed{2} \centerdot
  \left(60\right) + \boxed{3} \centerdot \left( 40 \right) + \boxed{5}
  \centerdot \left( 24 \right)
  \\
  \hline G_{33} & \slfrac{243}{20} & \boxed{2} \centerdot \left(60 \right)
  + \boxed{3} \centerdot \left(80 \right)
  \\
  \hline G_{34} & \slfrac{2401}{30} & \boxed{2} \centerdot \left(270
  \right) + \boxed{3} \centerdot \left(240 \right)
  \\
  \hline G_{35} \ (=E_6) & \slfrac{576}{5} & \boxed{2} \centerdot \left(90
  \right) + \boxed{3} \centerdot \left(60 \right)
  \\
  \hline G_{36} \ (=E_7) & \slfrac{19683}{14} & \boxed{2} \centerdot
  \left(210 \right) + \boxed{3} \centerdot \left(112 \right)
  \\
  \hline G_{37} \ (=E_8) & \slfrac{1265625}{56} & \boxed{2} \centerdot
  \left(504 \right) + \boxed{3} \centerdot \left(224 \right) \\
  \hlinewd{1pt}
\end{disarray}}
$  
}
}
\caption{Factorization of the $\LL$-discriminant for irreducible
  well-generated groups}\label{tab}
\end{table}

Here we detail explicit numerical data regarding the factorization of
the discriminant polynomial $D_{\LL}$.

Let us write (as in (\ref{eq:disc}) in Sec.\ \ref{subsec:LLwell})
\[ D_{\LL} = \prod_{i=1}^r D_i^{p_i} \] for the factorization of
$D_{\LL}$ into irreducible polynomials of $\BC[f_1,\dots, f_{n-1}]$.

\medskip

In Table \ref{tab}, we give, for each irreducible well-generated
group, the weighted degrees $\deg(D_i)$ and the powers $p_i$ which
appear in the factorization above. It is enough to deal with the
$2$-reflection groups, because any irreducible complex reflection
group is isodiscriminantal to a $2$-reflection group (see
\cite[Thm.2.2]{BessisKPi1}): it has the same discriminant $\Delta$,
and consequently the same braid group and the same polynomial
$D_{\LL}$. Thus, we only have to treat the four infinite series $A_n$,
$B_n$, $I_2(e)$, $G(e,e,n)$ (containing $D_n$), and $11$ exceptional
types (including the $6$ exceptional Coxeter groups).

\subsection*{Notations} 

In the last column of Table \ref{tab}, the ``$\LL$-data'':
\[ \boxed{p_1} \centerdot \left( u_1\right) + \boxed{p_2} \centerdot
\left( u_2\right) + \dots + \boxed{p_r} \centerdot \left(
  u_r\right) \] means that the form of the factorization is $ D_{\LL}
= \prod_{i=1}^r D_i^{p_i}$ with $\deg D_i = u_i$. This writing
reflects the additive decomposition of $\deg D_{\LL} = n(n-1)h$ (where
$n=\rk(W)$ and $h=d_n$) in terms of the $u_i$'s:
\[ \deg D_{\LL} = \sum_i p_i u_i \ . \]

\subsection*{By-products}

These numbers $(p_i,u_i)$ have many combinatorial interpretations. In
particular, thanks to Theorems \ref{thm:LL} and \ref{thm:fact}, we have:
\begin{itemize}
\item the number of conjugacy classes of parabolic Coxeter elements of
  length $2$ is the number of terms in the sum (each term
  $(p_i,u_i)$ of the sum
  corresponds to one of these classes, say $\Lambda_i$);
\item the order of the elements in $\Lambda_i$ is $p_i$ (provided $W$ is a
  $2$-reflection group);
\item the number $|\Fact_{(2,1,\dots, 1)}^{\Lambda_i}(c)|$ of
  submaximal factorizations of a Coxeter element $c$, whose first
  factor is in the class $\Lambda_i$, equals $\frac{(n-2)!\
    h^{n-1}}{|W|} u_i $. For convenience the first factor is also
  listed in the table, in the second colum.
\end{itemize}

We refer to \cite[App.\ A]{These} for a detailed explanation of the
computation of these data.

\section*{Acknowledgements}

I am grateful to Kyoji Saito for enriching mathematical exchanges
during the ``Hyperplane arrangements'' conference in Sapporo, in
August 2009. This work was part of my PhD thesis dissertation
\cite{These}. I would like to thank heartily my advisor David Bessis
for his constant support and his help on many points.

\bibliographystyle{alpha}
\bibliography{totalbibli}

\end{document}